\documentclass[12pt]{amsart}
\usepackage{amscd,amsfonts,amsmath,amssymb,amsthm,amstext,latexsym}
\usepackage{enumerate,pifont,graphicx}
\usepackage[english]{babel}
\usepackage[T1]{fontenc}
\usepackage{epsfig}
\usepackage[dvips]{color}
\def\a{{\alpha}}
\def\e{{\epsilon}}
\def\g{{\gamma}}

\def\H{{\mathbb H}}

\def\esf{\mathbb{S}}

\def\t{{\theta}}

\def\ve{{\varepsilon}}
\def\e{{\varepsilon}}

\newcommand{\R}{{\mathbb R}}

\newcommand{\N}{{\mathbb N}}
\newcommand{\M}{{\mathbb M}}

\newtheorem{theorem}{Theorem}[section]
\newtheorem{lemma}[theorem]{Lemma}
\newtheorem{proposition}[theorem]{Proposition}
\newtheorem{remark}[theorem]{Remark}
\newtheorem{corollary}[theorem]{Corollary}
\newtheorem{claim}[theorem]{Claim}

\begin{document}
\title{Classification of rotational special Weingarten surfaces of
  minimal type in ${\mathbb S}^2 \times \R$ and ${\mathbb H}^2
\times
  \R$}
\begin{abstract}
  In this paper we finish the classification of rotational special
  Weingarten surfaces in ${\mathbb S}^2 \times \R$
  and ${\mathbb H}^2\times \R$; i.e.  rotational surfaces in
  $\esf^2\times\R$ and $\H^2\times\R$ whose mean curvature $H$ and
  extrinsic curvature $K_e$ satisfy $H=f(H^2-K_e)$, for some function
  $f\in {\mathcal C}^1([0,+\infty))$ such that
  $4x(f'(x))^2<1$ for any $x\geq 0$.
\end{abstract}

\author{Filippo Morabito}
\author{M. Magdalena Rodr\'\i guez}
\thanks{This work was started during the first author
stay at Instituto de
  Matem\'atica Interdisciplinar of Universidad Complutense de
  Madrid. The second author is partially supported by a Spanish
  MEC-FEDER Grant no. MTM2007-61775 and a Regional J.
  Andaluc\'\i a Grant no. Grant P09-FQM-5088.}
\address{Instituto de Matem\'atica
Interdisciplinar, Universidad Complutense de Madrid, Plaza de las
  Ciencias 3, 28040, Madrid, Spain}
\address{Laboratoire de Mathématiques et Physique
Théorique UMR CNRS 6083, Université François Rabelais,
Parc Grandmont, 37200 Tours, France}
\email{morabitf@gmail.com}
\address{Departamento de Geometr\'\i a y Topolog\'\i a, Universidad de
  Granada, Spain}
\email{magdarp@ugr.es}

\maketitle

\section*{Introduction}
An oriented Riemannian surface $\Sigma$ in a 3-manifold $M$ is
called a {\it special Weingarten surface} if there exists $f\in
{\mathcal
  C}^1([0,+\infty))$ such that
\begin{equation}\label{eq.intro}
  H=f(H^2-K_e),
\end{equation}
where $H$ and $K_e$ denote respectively the mean curvature and the
extrinsic curvature of~$\Sigma$, and $f$ satisfies
\begin{equation}\label{eq:eliptica}
4x(f'(x))^2<1,\quad \mbox{for any }\ x\geq 0 .
\end{equation}
When $f(0)\neq 0$, the special Weingarten surfaces are called {\it
of constant mean curvature type}; and they are called {\it of
minimal type} when $f(0)=0$. Observe that, when $f$ is constant, we
get constant mean curvature surfaces (minimal surfaces in the case
the constant is $0$).

The study of Weingarten surfaces started with H. Hopf~\cite{H}, P.
Hartman and W. Wintner~\cite{HW} and S. S. Chern~\cite{C}, who
considered compact Weingarten surfaces in $\R^3$. More recently E.
Toubiana and R. Sa Earp~\cite{ST1,ST,ST2}), studied rotational
special Weingarten surfaces in $\R^3$ and ${\mathbb H}^3$. In the
case $f(0)\neq 0$ (mean curvature type) they determined necessary
and sufficient conditions for  existence and  uniqueness of
examples whose geometrical behaviour is the same as the one of
Delaunay surfaces in $\R^3$, i.e.  unduloids (embedded) and nodoids
(non-embedded), which have non-zero constant mean curvature. In the
case $f(0)=0$ (minimal type) they estabilished the existence of
examples whose geometric behaviour is the same as the one of the
catenoid of $\R^3$, which is the only rotational minimal surface in
$\R^3$.

H. Rosenberg and R. Sa Earp \cite{RS} showed that compact special
Weingarten surfaces in $\R^3$ and $\H^3$ satisfy an a priori height
estimate, and used this fact to prove that the annular ends of a
special Weingarten surface $M$ are cylindrically bounded. Moreover,
if such a $M$ is non-compact and has finite topological type, then $M$
must have more than one end; if $M$ has two ends, then it must be a
rotational surface; and if $M$ has three ends, it is contained in a
slab. They followed the ideas by Meeks~\cite{M} and
Korevaar-Kusner-Solomon~\cite{KKS} for non-zero constant mean
curvature surfaces in $\R^3$.

In \cite{FM}, the first author determines necessary and sufficient
conditions for existence and uniqueness of rotational special
Weingarten surfaces in $\esf^2\times\R$ and ${\mathbb H}^2 \times
\R$ of constant mean curvature type ($f(0)\neq 0$).  In this paper
we establish similar results in the minimal type case ($f(0)=0$),
finishing the classification of the rotational special Weingarten
surfaces in $\esf^2\times\R$ and $\H^2\times\R$.

We finally remark that the results studied in this paper generalize
already known theorems for minimal surfaces~\cite{HH,PR,NR,NT}.

\section{Preliminaires}
Throughout this paper, all the surfaces are assume to be ${\mathcal
C}^2$, immersed, connected and orientable. In this section we 
remind
some general results about special Weingarten surfaces in the
product manifold ${\mathbb M}\times \R$, where $\M$ is a Riemannian
surface with constant sectional curvature (after passing to the
universal covering, $\M=\R^2,\H^2$ or $\esf^2$).

Let $\Sigma$ be a special Weingarten surface in ${\mathbb
M}\times\R$; i.e. $\Sigma$ verifies equation~\eqref{eq.intro},
\[
H=f(H^2-K_e)
\]
for some elliptic function $f\in {\mathcal C}^1([0,+\infty))$, i.e.
$f$ satisfies equation~\eqref{eq:eliptica}. Let $F$ denote the
immersion of $\Sigma$ in $\M\times\R$, and take a domain
$D\subset\Sigma$ with compact closure and smooth boundary. Consider
any normal variation of $D$ given by a differentiable map
$\psi:(-\e,\e)\times\Sigma\to\M\times\R$, with $\e>0$, such that
$\psi(0,p)=F(p)$ for any $p\in\Sigma$; $\psi(s,p)=F(p)$ for any
$p\in\Sigma-D$ and any $|s|<\e$; and the map $\psi_s:\Sigma\to
\M\times\R$ defined by $\psi_s(p)=\psi(s,p)$ is an immersion for any
$|s|<\e$. Call, respectively, by $H(s)$ and $K_e(s)$ the mean
curvature and the extrinsic curvature of $\psi_s(\Sigma)$; $H(0)=H$
and $K_e(0)=K_e$. The first variation formula of
$H(s)-f(H(s)^2-K_e(s))$ is given by
\[
\left.\frac{d}{ds}\right|_{s=0}\left[H(s)-f(H(s)^2-K_e(s))\right]=
\left(1-2Hf'(H^2-K_e)\right)H'(0)+f'(H^2-K_e)K_e'(0).
\]
Elbert~\cite{E} proved that the principal parts of $H'(0)$ and
$K_e'(0)$ are respectively $\frac 1 2 \Delta$ and $L$, where $\Delta$
is the Laplacian with respect to the induced metric on $\Sigma$ and
$L$ is the operator given by
\[
Lu= div\left(T(\nabla u)\right),
\]
with $T(X)=2HX-A(X)$, for any tangent  vector $X$. Here $A$
denotes the shape operator of $\Sigma$.
So the linearized operator of $H(s)-f(H(s)^2-K_e(s))$ is given by
\[
L_f=\left( \frac{1-2Hf'}{2} \right)\Delta +f' L .
\]
As $f$ is elliptic (i.e. $4x(f'(x))^2<1$ for any $x\geq 0$), then
the eigenvalues of the operator $L_f$ are positive (see~\cite{RS},
page 294).  Thus the operator $L_f$ is elliptic, and
equation~\eqref{eq.intro} is elliptic in the sense of Hopf~\cite{H}.
Hence the solutions of~\eqref{eq.intro} satisfy an interior and a
boundary maximum principles. Let $\Sigma_1,\Sigma_2$ be two oriented
special Weingarten surfaces in $\M\times\R$
satisfying ~\eqref{eq.intro} for the same function $f$,
 whose unit normal
vectors coincide at a common point $p$. For $i=1,2$, we can write
$\Sigma_i$ locally around $p$ as a graph of a function $u_i$ over a
domain in $T_p\Sigma_1=T_p\Sigma_2$ (in exponential coordinates). We
will say $\Sigma_1$ is above $\Sigma_2$ at $p$, and we will write
$\Sigma_1\geq\Sigma_2$, if $u_1\geq u_2$.

\begin{proposition}[Maximum principle \cite{H}]
  \label{prop:maximum}
  Let $\Sigma_1,\Sigma_2$ be two special Weingarten surfaces
  in ${\mathbb M}\times \R$ with respect to the same elliptic function
  $f$.  Let us suppose that
  \begin{itemize}
  \item $\Sigma_1$ and $\Sigma_2$ are tangent at an interior
    point $p\in\Sigma_1\cap\Sigma_2$; or
  \item there exists $p\in\partial\Sigma_1\cap\partial\Sigma_2$
    such that both $T_p\Sigma_1=T_p\Sigma_2$ and
    $T_p\partial\Sigma_1=T_p\partial\Sigma_2$.
  \end{itemize}
  Also suppose that the unit normal vectors of $\Sigma_1,\Sigma_2$
  coincide at $p$. If $\Sigma_1\geq\Sigma_2$ at $p$, then
  $\Sigma_1=\Sigma_2$ in a neighborhood of $p$. In the case
  $\Sigma_1,\Sigma_2$ have no boundary, then $\Sigma_1=\Sigma_2$.
\end{proposition}

We will consider special Weingarten surfaces of minimal type; i.e.
we assume $f(0)=0$. Observe the first examples we get are horizontal
slices $\M\times\{t_0\}$, whose principal curvatures vanish
identically.

We call {\it height function} of $\Sigma$ to the restriction to
$\Sigma$ of the horizontal projection of $\M\times\R$ over $\R$.

\begin{corollary}\label{cor:t}
  Let $\Sigma$ be a rotational special Weingarten surface of minimal
  type in ${\mathbb M}\times \R$. If the height function of $\Sigma$
  has either a local maximum or a local minimum at an interior point
  $p$ of $\Sigma$, then $\Sigma$ is contained in the horizontal slice
  passing through $p$.
\end{corollary}
\begin{proof}
  It suffices to apply the maximum principle
  (Proposition~\ref{prop:maximum}) to $\Sigma$ and the corresponding
  horizontal slice.
\end{proof}

As a consequence, we get the following ``halfspace-type theorem'' in
$\esf^2\times\R$.

\begin{corollary}\label{cor:half}
  If $\Sigma\subset\esf^2\times\R$ is a rotational special Weingarten
  surface of minimal type contained in a horizontal halfspace
  $\esf^2\times[t_0,+\infty)$ or $\esf^2\times(-\infty,t_0]$, for some
  $t_0\in\R$, then $\Sigma$ is contained in a horizontal slice.
\end{corollary}

Theorem~\ref{th:h2} says Corollary~\ref{cor:half} does not hold in
$\H^2\times\R$, since there are examples of catenoidal type
contained in horizontal slabs (also see~\cite{NR, NT}).

\medskip

The following proposition shows one of the
common aspects
of the theory
of classical minimal surfaces and the theory of special Weingarten
surfaces of minimal type.

\begin{proposition}
  \label{prop:sign}
  The extrinsic curvature $K_e$ of a special Weingarten surface
  $\Sigma$ of minimal type in $\M\times\R$ is non-positive. Moreover
  $K_e=0$ if, and only if, both principal curvatures vanish identically.
\end{proposition}
\begin{proof}
  Let $k_1$ and $k_2$ denote the principal curvatures of
  $\Sigma$. Then $H^2-K_e=\frac{(k_1-k_2)^2}{4}$.

  If we consider $g(x)=x-f(x^2)$, the Weingarten surface
  equation~\eqref{eq.intro} can be rewritten as
  \begin{equation}
    \label{eq:k}
    g\left(\frac{k_2-k_1}{2}\right)=-k_1  \qquad\mbox{or}\qquad
    g\left(\frac{k_1-k_2}{2}\right)=-k_2 .
  \end{equation}
  Using Lemma \ref{lem:increasing} below, we get that $g$ is a
  strictly increasing function which only vanishes at $x=0$.  Then we
  directly deduce from~\eqref{eq:k} that, given $p\in\Sigma$:
  \begin{itemize}
  \item $k_1(p)>k_2(p)$ if, and only if, $k_1(p)>0>k_2(p)$.
  \item $k_1(p)<k_2(p)$ if, and only if, $k_1(p)<0<k_2(p)$.
  \item $k_1(p)=k_2(p)$ if, and only if, $k_1(p)=0=k_2(p)$.
  \end{itemize}
  This proves Proposition~\ref{prop:sign}.
\end{proof}

\begin{remark}
  Proposition~\ref{prop:sign} holds in an arbitrary 3-manifold.
\end{remark}

Finally, let us prove the following technical Lemma that we will use
throughout the paper.

\begin{lemma} \label{lem:increasing}
 Given $f \in {\mathcal C}^1([0,+\infty))$, the following statements
 are equivalent:
 \begin{enumerate}
 \item $f$ is elliptic and $f(0)=0$.
 \item $g(x)=x-f(x^2)$ is a strictly increasing function with $g(0)=0$.
 \item $\bar g(x)=x+f(x^2)$ is strictly increasing and $\bar
   g(0)=0$.
 \end{enumerate}
\end{lemma}
\begin{proof}
  Let us prove {\it (1)} $\Rightarrow$ {\it (2)}.  Suppose $f$ is
  elliptic.  Considering $x^2$ instead of $x$, we deduce that the
  ellipticity of $f$ is equivalent to
  \[
  -1<2xf'(x^2)<1, \quad\mbox{for any }\ x\in\R .
  \]
  Hence $g'(x)=1-2 xf'(x^2)>0$, and $g$ is a strictly increasing
  function. Finally, it is clear that $g(0)=-f(0)=0$.

  Conversely, suppose $g'(x)=1-2 xf'(x^2)>0$ for any $x\in\R$; i.e.
  $2xf'(x^2)<1$ for any $x\in\R$. Since this holds for 
  positive and
  negative values of $x$, we get $|2xf'(x^2)|<1$. And then $4t
  (f'(t))^2<1$, $\forall t\geq 0$. Moreover, $f(0)=-g(0)$. This
  proves {\it (2)} $\Rightarrow$ {\it (1)}.

  {\it (1)} $\Leftrightarrow$ {\it (3)} can be proved similarly.
\end{proof}

\begin{corollary}\label{cor:limit}
  Let $f \in {\mathcal C}^1([0,+\infty))$ be an elliptic function such
  that $f(0)=0$. Then there exist
  \[
  \ell_{-\infty}=\lim_{r \to -\infty}(r-f(r^2))\quad\mbox{and
  }\quad\ell_{+\infty}=\lim_{r\to+\infty}(r-f(r^2)),
  \]
  with $\ell_{-\infty}\in[-\infty,0)$ and $\ell_{+\infty}\in(0,+\infty]$.
\end{corollary}

\section{Rotational special Weingarten surfaces of minimal type}
\label{sec:H2}

Given $\ve\in\{1,-1\}$, $\M_\ve$ will denote the sphere $\esf^2$, when
$\ve=1$, or the hyperbolic plane $\H^2$, when $\ve=-1$.

We consider in $\esf^2=\{(x_1,x_2,x_3)\in\R^3\ |\
x_1^2+x_2^2+x_3^2=1\}$ the usual metric $dx_1^2+dx_2^2+dx_3^2$ induced
from $\R^3$;  and we will see $\H^2$ as a subvariety of $\mathbb{L}^3$, this is
\[
\H^2=\{(x_1,x_2,x_3)\in\R^3\ |\ x_1^2+x_2^2-x_3^2=-1,\ x_3>0\}
\]
with the induced metric $dx_1^2+dx_2^2-dx_3^2$.

Define
\[
S_\ve(x)=\left\{\begin{array}{ll}
    \sin(x), & \mbox{when } \ve=1  \\
    \sinh(x), & \mbox{when } \ve=-1
    \end{array}\right.
\]
\[
C_\ve(x)=\left\{\begin{array}{ll}
    \cos(x), & \mbox{when } \ve=1  \\
    \cosh(x), & \mbox{when } \ve=-1
    \end{array}\right.
\]
for $x\in I_\ve$, where $I_\ve=[0,\pi]$ if $\ve=1$ and
$I_\ve=[0,+\infty)$ if $\ve=-1$. We consider $\M_\ve\times\R$
parameterized as
\[
\{(S_\ve(\phi)\, \cos\theta,S_\ve(\phi)\,
\sin\theta,C_\ve(\phi),t)\ |\ (\phi,\theta,t)\in
I_\ve\times[0,2 \pi)\times\R\}.
\]

Let $\Sigma_\g$ be a rotational surface in $\M_\ve\times\R$
obtained by rotating a curve
\[
\g(s)=(S_\ve(\phi(s)), 0,C_\ve(\phi(s)), t(s)),\quad s\in I\subset\R ,
\]
around the axis $\{(0,0,1)\}\times\R$, where $I$ is an open interval.
The surface $\Sigma_\g$ is then parameterized by
\[
F(s,\t)=\left(S_\ve(\phi(s))\cos\theta, S_\ve(\phi(s))\sin\theta,
  C_\ve(\phi(s)), t(s)\right),\quad s\in I\subset\R,\ \t\in[0,2\pi).
\]
Assume $s$ is the arc-length parameter of $\g$; i.e.
$\phi'(s)^2+t'(s)^2=1$.  The principal curvatures of $\Sigma_\g$ at
the point $F(s,\t)$ with respect to the unit normal vector field
\begin{equation}\label{eq:normal}
N(s,\t)= \Big(t'(s) C_\ve(\phi(s)) \cos\t,\, t'(s)
C_\ve(\phi(s))\sin\t,\, -t'(s)S_\ve(\phi(s)),\, -\phi'(s)\Big)
\end{equation}
are given by
\[
k_1(s)=t''(s) \phi'(s)-t'(s) \phi''(s) \quad\mbox{and}\quad
k_2(s)=t'(s)\eta_\ve(\phi(s)) ,
\]
where
\[
\eta_\ve(x)=\left\{\begin{array}{ll}
    \cot(x),  & \mbox{when } \ve=1 ,\\
    \coth(x), & \mbox{when } \ve=-1 .
    \end{array}\right.
\]
Hence the Weingarten surface equation~\eqref{eq.intro} becomes
\begin{equation}\label{eq:W}
  \frac{t''\phi'-t'\phi''+t'\eta_\ve(\phi)}{2}=
  f\left(\frac{(t''\phi'-t'\phi''-t'\eta_\ve(\phi))^2}{4}\right).
\end{equation}
Remind $f$ denotes a fixed elliptic function such that $f(0)=0$.

\begin{remark}
  Observe that, if $\phi(s),t(s)$ solve equation \eqref{eq:W}, then
  $\phi(s),\widetilde t(s)=t(s)+t_0$ also do, for any $t_0\in\R$. In
  particular, we can identify rotational special Weingarten surfaces
  of minimal type in $\M_\ve\times\R$ by vertical translations in the
  direction of $\R$.
\end{remark}

It is clear that $\phi(s)=s,t(s)=t_0$ solve equation~\eqref{eq:W} for
any $t_0\in\R$.  We then check the horizontal slices
$\M_\ve\times\{t_0\}$ are particular cases of rotational special
Weingarten surfaces of minimal type in $\M_\ve\times\R$, as we already
knew.

We deduce from Corollary~\ref{cor:t} that, if $\Sigma_\g$ is not
contained in a horizontal slice, then $t$ cannot have a local maximum
nor a local minimum. In particular, either $t'\geq 0$ or $t'\leq 0$.
The following Lemma shows the inequalities are strict.

\begin{lemma}\label{lem:no0}
  Let $\Sigma_\g$ be a rotational special Weingarten surface of
  minimal type in $\M_\ve\times\R$ which is not contained in a
  horizontal slice. Then $t'$ never vanishes at an interior point of
  $\Sigma_\g$.
\end{lemma}
\begin{proof}
  Suppose there exists $s_0\in I$ such that $t'(s_0)=0$. Since
  $\phi'(s)^2+t'(s)^2=1$, then we have $\phi'(s_0)=\pm 1$. In
  particular, $\phi'(s)\neq 0$ in some interval
  $J=(s_0-\delta,s_0+\delta)$. Thus, using $t't''+\phi'\phi''=0$, we
  get $k_1(s)=t''(s) \phi'(s)-t'(s) \phi''(s)=\frac{t''(s)}{\phi'(s)}$
  for any $s\in J$, if we consider $\Sigma_\g$ oriented
  by~\eqref{eq:normal}; and then the Weingarten surface
  equation~\eqref{eq:W} becomes
  \[
  \frac{t''+\phi't'\eta_\ve(\phi)}{2\phi'}=
  f\left(\frac{(t''-\phi't'\eta_\ve(\phi))^2}{4(\phi')^2}\right)
  \]
  in $J$.  This equation can be rewritten as $\widetilde
  G(\phi,\phi',t',t'')=0$, where
    \[
  \widetilde G(x,y,z,w):= \frac{w+yz\eta_\ve(x)}{2y}-
  f\left(\frac{(w-yz\eta_\ve(x))^2}{4y^2}\right).
  \]
  \begin{claim}\label{cl:G}
    $\widetilde G(x,y,z,w)$ is strictly increasing (resp. strictly
    decreasing) with respect to~$w$, when restricted to $\{y>0\}$
    (resp. $\{y<0\}$).
  \end{claim}
  \noindent
  Let us prove Claim~\ref{cl:G}.  A straightforward computation gives
  \[
  \frac{\partial \widetilde G}{\partial w}= \frac 1{2y}(1-2\a
  f'(\a^2)),\quad \mbox{where }\ \a=\frac{w-yz\eta_\ve(x)}{2y}.
  \]
  Since $f$ is elliptic, we get $1-2\a f'(\a^2)>0$, from where the
  claim follows.

  \medskip

  Since $\phi'(s_0)=\pm 1$, the claim above says $\frac{\partial
    \widetilde G}{\partial
    w}(\phi(s_0),\phi'(s_0),t'(s_0),t''(s_0))\neq 0$. Using the
  implicit function theorem, we get there exists a ${\mathcal C}^1$
  function $\widetilde\Upsilon$ from a neighborhood of
  $(\phi(s_0),\phi'(s_0),t'(s_0))$ in $\R^3$ to a neighborhood of
  $t''(s_0)$ in $\R$ such that $t''=\widetilde\Upsilon(\phi,\phi',t')$
  in $J$ (considering a smaller $\delta$ if necessary). If we set
  $v_1=\phi,v_2=\phi',v_3=t,v_4=t'$, the Weingarten surface equation
  $\widetilde G(\phi,\phi',t',t'')=0$ in $J$ becomes
  \begin{equation}\label{eq:system}
    \left\{\begin{array}{l}
        v_1'=v_2\\
        v_2'=-\frac{v_4}{v_2} \widetilde \Upsilon(v_1,v_2,v_4)\\
        v_3'=v_4\\
        v_4'=\widetilde \Upsilon(v_1,v_2,v_4).
      \end{array}\right.
  \end{equation}
  Given initial values $\phi(s_0)=\phi_0, \phi'(s_0)=\pm 1,
  t(s_0)=t_0, t'(s_0)=0$, Picard-Lindel\"{o}f theorem gives existence
  and uniqueness of a solution $(\phi,t)$ for~\eqref{eq:system} with
  these initial values at $s_0$, defined in $J$ (possibly taking a
  smaller $\delta$). By uniqueness,
  \[
  \phi(s)=\phi_0\pm(s-s_0)\quad\mbox{and}\quad t(s)=t_0,
  \]
  for any $s\in J$.  By the maximum principle,
  $\Sigma_\g\subset\M\times\{t_0\}$, a contradiction.
\end{proof}

From now on, we will assume $\Sigma_\g$ is a rotational special
Weingarten surface of minimal type in $\M_\ve\times\R$ which is not
contained in a horizontal slice, and $\Sigma_\g$ is oriented by the
unit normal vector field~\eqref{eq:normal}.

Since  $t'\neq 0$, using $t't''+\phi'\phi''=0$
(which holds because $(\phi')^2+(t')^2=1$) we get
\[
k_1(s)=-\frac{\phi''(s)}{t'(s)}
\]
for any $s\in I$. And then the Weingarten surface
equation~\eqref{eq:W} becomes
\begin{equation}\label{eq:W2}
  \frac{(t')^2\eta_\ve(\phi)-\phi''}{2t'}=
  f\left(\frac{((t')^2\eta_\ve(\phi)+\phi'')^2}{4(t')^2}\right) .
\end{equation}

\begin{lemma}
  \label{lem:cyl}
  Let $\Sigma_\g\subset\M_\ve\times\R$ be a rotational special
  Weingarten surface of minimal type as above.  If $\Sigma_\g$ lies in
  a vertical cylinder, then $\ve=1$ (i.e. $\M_\ve=\esf^2$) and such a
  vertical cylinder is $\{\phi=\pi/2\}$.
\end{lemma}
\begin{proof}
  If $\Sigma_\g$ is contained in a vertical cylinder, then
  $\phi(s)=\phi_0$ and $t(s)=\pm s$, for some $\phi_0\in
  I_\ve$. Equation~\eqref{eq:W2} becomes
  \[
  \pm \frac{\eta_\ve(\phi_0)}{2} =
  f\left(\frac{\eta_\ve(\phi_0)^2}{4}\right) .
  \]
  By Lemma \ref{lem:increasing} we know the only zero of the function
  $g(x)=x-f(x^2)$ is $x=0$. Thus it must hold $\eta_\ve(\phi_0)=0$,
  which is only possible when $\ve=1$ and $\phi_0=\pi/2$.
\end{proof}

By Proposition~\ref{prop:maximum}, $\Sigma_\g$ cannot touch the axis
of revolution $\{\phi=0\}$: note that if $\g$ cuts the axis, it must
be orthogonally (as $\Sigma_\g$ is a regular surface) and we reach a
contradiction by applying the maximum principle to $\Sigma_\g$ and the
corresponding horizontal slice.  Similarly, $\g$ can neither touch the
line $\{\phi=\pi\}$ in the case $\M_\ve=\esf^2$.

Proposition~\ref{prop:maximum} also says that, if:
\begin{itemize}
\item $\Sigma_\g$ has no boundary,
\item and $\Sigma_\g$ is not a horizontal slice, in the case
  $\M_\ve=\esf^2$,
\end{itemize}
then it cannot be compact, and the generating curve $\g$ must be
defined in $I=\R$.  Remark that in the case $\M_\ve=\esf^2$ it must be
$t(\R)=\R$, by Corollary~\ref{cor:half}.

\medskip

If we set
\begin{equation}\label{eq:G}
G(x,y,z):= \frac{y^2\eta_\ve(x)-z}{2y}
-f\left(\frac{(y^2\eta_\ve(x)+z)^2}{4y^2}\right),
\end{equation}
it is clear that \eqref{eq:W2} can be rewritten as
$G(\phi,t',\phi'')=0$.

Similarly as Claim~\ref{cl:G}, it can be proven the following one.
\begin{claim}\label{lem:G}
  $G(x,y,z)$ is strictly increasing (resp. strictly decreasing)
  with respect to the $z$ variable when restricted to $\{y<0\}$
  (resp. $\{y>0\}$).
\end{claim}

Fix $s_0\in I$. By the implicit function theorem and Claim~\ref{lem:G}
there exists a ${\mathcal C}^1$ function $\Upsilon$ defined in a
neighborhood of $(\phi(s_0),t'(s_0))$ in $\R^2$ such that
\[
\phi''=\Upsilon(\phi,t')
\]
in a neighborhood of $s_0$. If we set
$v_1=\phi,v_2=\phi',v_3=t,v_4=t'$, the Weingarten surface
equation~\eqref{eq:W2} is equivalent to the system
\begin{equation}\label{eq:system.2}
  \left\{\begin{array}{l}
      v_1'=v_2\\
      v_2'=\Upsilon(v_1,v_4)\\
      v_3'=v_4\\
      v_4'=-\frac{v_2}{v_4} \Upsilon(v_1,v_4).
      \end{array}\right.
\end{equation}
Given initial values $\phi(s_0), \phi'(s_0), t(s_0),t'(s_0)\in\R$,
with $\phi'(s_0)^2+t'(s_0)^2=1$ and $t'(s_0)\neq 0$,
Picard-Lindel\"{o}f Theorem gives existence and uniqueness of a
solution $(\phi,t)$ for~\eqref{eq:system.2} with the fixed initial
values at~$s_0$, defined in some interval $(s_0-\delta,s_0+\delta)$
where $t'\neq 0$.  Remark that $v_2 v_2'+v_4 v_4'=0$, and then
$(\phi')^2+(t')^2$ is constant. Since $\phi'(s_0)^2+t'(s_0)^2=1$, we
get $(\phi')^2+(t')^2=1$.

In the following Lemma, we use the uniqueness of such a solution
$(\phi,t)$ to get symmetries of the examples.

\begin{lemma}
  \label{symmetry}
  Let $\Sigma_\g\subset\M_\ve\times\R$ be a rotational special
  Weingarten surface of minimal type as above.
  Suppose there exists $s_0\in\R$ and $\delta>0$ (possibly
  $\delta=+\infty$) such that $\phi$ is defined in
  $I=(s_0-\delta,s_0+\delta)$ and $\phi'(s_0)=0$.  Then $\Sigma_\g$ is
  symmetric with respect to the horizontal slice $\{t=t(s_0)\}$; more
  precisely, $\phi(s)=\phi(2s_0-s)$ and $t(s)=2t(s_0)-t(2s_0-s)$, for
  any $s\in I$.
\end{lemma}
\begin{proof}
  Since $\phi'(s)^2+t'(s)^2=1$ and $\phi'(s_0)=0$, it follows
  $t'(s_0)=\pm 1$.  Define
  \[
  \psi(s)=\phi(2s_0-s)\qquad\mbox{ and }\qquad h(s)=2t(s_0)-t(2s_0-s)
  \]
  for $s\in I$.  Then $h'(s_0)=t'(s_0)=\pm1$.  In particular, $h'\neq
  0$ in a small neighborhood of $s_0$, where it is easy to check that
  $G(\psi,h',\psi'')=0$.  Since $\psi(s_0)=\phi(s_0)$,
  $\psi'(s_0)=\phi'(s_0)$, $h(s_0)=t(s_0)$ and $h'(s_0)=t'(s_0)$, we
  deduce from the uniqueness of solution that $\psi(s)=\phi(s)$ and
  $h(s)=t(s)$ locally around $s_0$.  The maximum principle gives the
  global equality.
\end{proof}

\begin{remark}\label{rem:sym}
  In the hypothesis of Lemma~\ref{symmetry}, if $\phi$ is defined in
  $(s_0-\delta_1,s_0+\delta_2)$ and $\phi'(s_0)=0$, then we can extend
  by symmetry the solution $(\phi,t)$ of the Weingarten surface
  equation to $(s_0-\delta_0,s_0+\delta_0)$, where
  $\delta_0=\max\{\delta_1,\delta_2\}$.
\end{remark}

The following Lemma shows that the function $\phi$ appearing in the
definition of the generating curve $\g$ has no horizontal asymptotes
when $\phi''$ has a sign.

\begin{lemma}\label{lem:asymp}
  Let $\Sigma_\g$ be a rotational special Weingarten surface of
  minimal type as above. Suppose that $\Sigma_\g$ is not contained in
  a vertical cylinder and either $\phi''(s)\leq 0$ for any $s\in I$
  or $\phi''(s)\geq 0$ for any $s\in I$.
  \begin{enumerate}
  \item If $I=(s_0,+\infty)$, then there exists $\lim_{s\to +\infty}
    \phi'(s)=\pm 1$.
  \item If $I=(-\infty,s_0)$, then there exists $\lim_{s\to -\infty}
    \phi'(s)=\pm 1$.
  \end{enumerate}
\end{lemma}

\begin{proof}
  In the case $\Sigma_\g$ is contained in a horizontal slice, then
  $\phi(s)=s$ and $t(s)=t_0$. Hence it is clear Lemma~\ref{lem:asymp}
  holds. Then suppose this is not the case.

  Let us first assume $I=(s_0,+\infty)$.  Since $\phi'$ is monotone
  and $\phi'^2(s)\leq 1$, there exists $\lim_{s\to +\infty}
  \phi'(s)\in[-1,1]$. Consider $s_1\in I$, and define $s_n=s_1+n-1\in
  I$, for any $n\in\N$.  By the intermediate value theorem, there
  exists $u_n\in(s_n,s_{n+1})$ such that
  \[
  \phi''(u_n)=\phi'(s_{n+1})-\phi'(s_n),
  \]
  which converges to zero as $n\to+\infty$. Hence either  $k_1(u_n)=-\frac{\phi''(u_n)}{t'(u_n)}\to 0$ 
  or $t'(u_n)\to 0$.   We will have
  finished the proof of {\it (1)} if we get $t'(u_n)\to 0$.

  Suppose $t'(u_n)\not\to 0$. Thus $k_1(u_n)\to 0$ and, by
  Proposition~\ref{prop:sign}, we deduce
  $k_2(u_n)=t'(u_n)\eta_\ve(\phi(u_n))\to 0$. This is not possible when
  $\M_\ve=\H^2$. Then it must be $\M_\ve=\esf^2$ and
  $\phi(u_n)\to\pi/2$. By the intermediate value theorem, there exists
  $v_n\in(u_{2n},u_{2n+2})$ such that
  \[
  (u_{2n+2}-u_{2n})\phi'(v_n)=\phi(u_{2n+2})-\phi(u_{2n}),
  \]
  which converges to zero. Since $u_{2n+2}-u_{2n}>
  s_{2n+2}-s_{2n+1}=1$, then we have $\phi'(v_n)\to 0$. This implies
  $\lim_{s\to +\infty} \phi'(s)=0$, and $\lim_{s\to +\infty}
  \phi(s)=\pi/2$.

  Remind that either $t'>0$ or $t'<0$.  If $\phi''\leq 0$ and $t'>0$,
  then $k_1=-\frac{\phi''}{t'}\geq 0$. On the other hand we have
  $\phi\leq\pi/2$, since $\phi''\leq 0$ and $\lim_{s\to +\infty}
  \phi(s)=\pi/2$. Then $k_2\geq 0$. By
  Proposition~\ref{prop:sign}, the only possibility is $k_1=k_2=0$
  identically in $I$. But this is not possible, as we are assuming
  $\Sigma_\g$ is not contained in a horizontal slice nor in the vertical
  cylinder $\{\phi=\pi/2\}$.

  The remaining three cases: $\phi''\leq 0$ and $t'<0$; $\phi''\geq 0$
  and $t'>0$; and $\phi''\geq 0$ and $t'<0$ follow analogously. This
  finishes the proof of item~{\it (1)}.

  Second item can be proved similarly.
\end{proof}


\subsection{Rotational special Weingarten surfaces of minimal type in
  ${\mathbb H}^2 \times \R$}
\label{sub:H2}

In \cite{FM},  first author studied rotational special Weingarten
surfaces of constant mean curvature type in $\H^2 \times \R$. He set
the necessary and sufficient conditions for the existence and
uniqueness of examples. When $f(0)>0$, we get embedded examples,
which either have the same geometric behaviour as the Delaunay
unduloid or they are vertical cylinders. If $f(0)<0$, the examples
we get are non-embedded examples, and they have a Delaunay nodoidal
type behaviour.

In the present paper we will prove that, when $f(0)=0$, the only
rotational special Weingarten surfaces of minimal type in $\H^2
\times \R$ are the horizontal slices $\H^2\times\{t_0\}$ and
properly embedded surfaces of catenoidal type.  In particular all
the examples are properly embedded.

\begin{lemma}\label{lem:sign}
  Let $\Sigma_\g$ be a rotational special Weingarten surface of
  minimal type in ${\mathbb H}^2 \times \R$, generated by
  $\g(s)=(S_\ve(\phi(s)), 0,C_\ve(\phi(s)), t(s))$, $s\in I$. If
  $\Sigma_\g$ is not contained in a horizontal
  slice, 
  then $\phi''> 0$. 
\end{lemma}
\begin{proof}
  By Lemma~\ref{lem:no0}, $t'\neq 0$.  Since
  $t'(s)k_2(s)=(t'(s))^2\coth(\phi(s))> 0$,
  Proposition~\ref{prop:sign} says that $-\phi''(s)=t'(s)k_1(s)<0$,
  from where Lemma~\ref{lem:sign} follows.
\end{proof}


\begin{theorem}\label{th:h2}
  Let $\Sigma_\g\subset{\mathbb H}^2\times\R$ be a rotational special
  Weingarten surface of minimal type without boundary, generated by
  $\g(s)=(S_\ve(\phi(s)),0,C_\ve(\phi(s)),t(s))$, for $s\in
  I\subset\R$.  Then $I=\R$ and either $\Sigma_\g$ is a horizontal
  slice or a catenoidal type surface whose functions $\phi,t$ satisfy:
  \begin{itemize}
  \item $\phi$ is a strictly convex function (i.e. $\phi''>0$), having
    a unique local minimum at some $s_0\in\R$ and
    $\lim_{s\to\pm\infty}\phi(s)=+\infty$.  Up to a
    reparameterization, we can assume $s_0=0$. Then
    $\phi(s)>\phi(0)>0$ for any $s\neq 0$. Moreover,
    $\phi(-s)=\phi(s)$, for any $s\in\R$.
  \item $t$ is a strictly monotone function. Up to a vertical
    translation of $\Sigma_\g$, we can assume $t(0)=0$. Then
    $t(-s)=-t(s)$, for any $s\in\R$.  Furthermore,
    $-t_\infty<t(s)<t_\infty$ for any $s\in\R$, where $t_\infty
    \in\R$ denotes $\lim_{s\to+\infty}t(s)$ if
    $t'>0$ or $\lim_{s\to-\infty}t(s)$ if $t'<0$.
  \end{itemize}
\end{theorem}

\begin{proof}
  First observe that $\g$ is defined in $\R$ because $\Sigma_\g$ has
  no boundary, since $\Sigma_\g$ cannot be compact. Suppose
  $\Sigma_\g$ is not a horizontal slice. Thus Lemma~\ref{lem:no0}
  ensures $t$ is strictly monotone.

  We get from Lemma~\ref{lem:sign} that $\phi''>0$, and then $\phi$ has no
  local maxima.  Moreover, as an application of the intermediate value
  theorem, we get by Lemma~\ref{lem:asymp} that $\phi$ cannot have a
  horizontal asymptote. Then $\phi$ has a unique local minimum at some
  $s_0 \in \R$ and $\lim_{s\to\pm\infty}\phi(s)=+\infty$.

  After a reparameterization, we can assume $s_0=0$. We can also
  assume $t(0)=0$, up to a vertical translation of~$\Sigma_\g$.
  Lemma~\ref{symmetry} says $\g$ is symmetric with respect to the
  horizontal slice $\H^2\times\{0\}$; more precisely,
  $\phi(-s)=\phi(s)$ and $t(-s)=-t(s)$, for any $s\in\R$.
  The existence of $t_\infty\in\R\cup\{+\infty\}$ and the wanted inequalities follow from
  the monotonicity of~$t$.

  It remains to prove $t_\infty<+\infty$. First we observe
  there exists $\delta>0$ small such that
  \[
  |f(x^2)|\leq \frac{|x|}{2},\quad \mbox{for any } |x|<\delta,
  \]
  because $f$ is a $\mathcal{C}^1$ function and $f(0)=0$.
  Since $(\phi')^2+(t')^2=1$, Lemma~\ref{lem:asymp} implies $t'(s)\to
  0$ as $s\to\pm\infty$. Then $k_2(s)=t'(s)\coth(\phi(s))\to 0$.  By
  Proposition~\ref{prop:sign}, we also get $k_1(s)\to 0$.  Hence, for
  $|s|$ big enough, we have 
  $\frac{|k_1(s)-k_2(s)|}{2}<\delta$,
  and then
  \begin{equation}\label{eq:desig}
  -\frac{|k_1(s)-k_2(s)|}{2}\leq
  k_1(s)+k_2(s)=2 f\left( \frac{(k_1(s)-k_2(s))^2}{4}\right)\leq
  \frac{|k_1(s)-k_2(s)|}{2}.
  \end{equation}

  In the case $t'(s)>0$, we have $k_1(s)<0<k_2(s)$ and then
  $|k_1(s)-k_2(s)|=k_2(s)-k_1(s)$.  From the second inequality
  in~\eqref{eq:desig}, we get $k_1(s)\leq
  \frac{-1}{3} k_2(s)$.  On the other hand, we get from
  Lemmas~\ref{lem:asymp} and~\ref{lem:sign} that
  $\phi'(s)\to 1$ as $s\to+\infty$.
  In particular, $\phi'(s)>0$ for $s$ large, and then
  $k_1(s)\leq  \frac{-1}{3} k_2(s)$ says
  \[
  t''(s)\leq \frac{-1}{3} t'(s)\coth(\phi(s)) \phi'(s).
  \]
  Since $\phi(s)\to+\infty$ as $s\to +\infty$, there exists
  a positive constant $b$ such that
  \[
  \frac{t''(s)}{t'(s)}\leq -b,\quad\mbox{for } s>s_0,
  \]
  where $s_0>0$ is large enough.  Integrating on
  $(s_0,s)$, we get $t'(s) \leq c e^{-bs}$, where $c=t'(s_0)e^{b s_0}$
  is a positive constant. Now integrating on $(s_0,+\infty)$ we get
  $t_\infty\leq\frac{t'(s_0)}{b} +t(s_0)<+\infty$, as we wanted to prove.

  Now suppose $t'(s)<0$.  Then  $k_1(s)>0>k_2(s)$ and, using the first inequality
  in~\eqref{eq:desig}, we obtain $k_1(s)\geq \frac{-1}{3} k_2(s)$.
  Arguing as above we get
  $\frac{t''(s)}{t'(s)}\leq -b$ for $s>s_0$, where $b$ is a positive
  constant and $s_0>0$ is large enough.
  Integrating on $(s_0,s),$ we obtain $t'(s)\geq -c
  e^{-bs}$, with $c$ a positive constant.  Finally, integrating on
  $(s_0,+\infty)$ we conclude that $-t_\infty=\lim_{s \to
    +\infty}t(s)>-\infty$.
    \end{proof}

%



\begin{remark}
  A minimal surface can be seen as a Weingarten surface for which $f$
  vanishes identically.  Theorem~\ref{th:h2}
  says in particular that any the rotational minimal surface of
  $\H^2\times\R$ lies in a horizontal slab. This fact was already
  obtained by B. Nelli, R. Sa Earp, W. Santos and E.
  Toubiana~\cite{NT}, Proposition 5.1.
\end{remark}

\begin{remark}
  In Theorems~\ref{th:exist1} and~\ref{th:exist2} we prove that, fixed
  the elliptic function $f$, there exists a unique rotational special
  Weingarten surface of minimal type in ${\mathbb H}^2 \times \R$ as
  in Theorem~\ref{th:h2} with $\phi(0)=\phi_0$, for any $\phi_0>0$
  satisfying some additional conditions (see
  Lemma~\ref{lem:-infinity}).
\end{remark}

\subsection{Rotational special Weingarten surfaces of minimal type in
  $\esf^2\times\R$}
\label{sub:S2} In \cite{FM}, the first author studied rotational
special Weingarten surfaces in ${\mathbb S}^2 \times \R$ with no
umbilical points.  He showed that the embedded examples are only
those who are strictly contained in $({\mathbb S}^2)^\pm \times \R,
$ where $({\mathbb S}^2)^\pm$ denotes the hemispheres of $\esf^2,$
and they are either vertical cylinders or Delaunay unduloidal type
surfaces.  Furthermore non-embedded examples exist only under the
condition $f(0)<0.$

In this subsection we will show that the rotational special Weingarten
surfaces of minimal type in $\esf^2\times\R$ are: the horizontal
slices $\esf^2 \times \{t_0\}$, the vertical cylinder $\{\phi=\pi/2\}$
and properly embedded surfaces of unduloidal type.  In particular they
are all properly embedded.

Throughout this subsection, $\Sigma_\g$ will denote a rotational
special Weingarten surface of minimal type in $\esf^2\times\R$,
obtained by rotating a curve
\[
\g(s)=(\sin(\phi(s)), 0,\cos(\phi(s)), t(s)) ,\quad s\in I,
\]
around the axis $\{(0,0,1)\}\times\R$. Assume $\Sigma_\g$ is not
contained in a horizontal slice nor in the cylinder $\{\phi=\pi/2\}$.

By Lemma~\ref{lem:no0}, either $t'>0$ or $t'<0$.  The following Lemma,
which can be proved arguing similarly as in Lemma~\ref{symmetry}, says
we do not lose generality assuming $t'>0$.
\begin{lemma}
  \label{another.solution}
  Let $\Sigma_\g$ be a rotational special Weingarten surface of
  minimal type as above, $\g$ defined in $I=(a,b)$, where
  $a\in[-\infty,+\infty)$ and $b\in(-\infty,+\infty]$.  Given $s_0\in
  I$, the functions
  \[
  \psi(s)=\pi-\phi(2s_0-s)\quad\mbox{ and }\quad v(s)=t(2s_0-s),
  \]
  for $s\in(2s_0-b,2s_0-a)$, also produce a rotational special
  Weingarten surface in the same hypothesis as $\Sigma_\g$.
\end{lemma}

From now on, assume $t'>0$; i.e. $t$ is a strictly increasing
function.

\medskip

We saw in Lemma~\ref{lem:sign} that $\phi''$ never vanishes when
$\Sigma_\g\subset\H^2\times\R$. Next Lemma says that, if we are
working in $\esf^2\times\R$, then $\phi''$ only vanishes when the
surface cuts the vertical cylinder $\{\phi=\pi/2\}$.

\begin{lemma}\label{lem:signS2}
  Let $\Sigma_\g$ be a rotational special Weingarten surface as
  above. Assume $\Sigma_\g$ is not contained in a slice.  Then
  $\phi''(s)=0$ if, and only if, $\phi(s)=\pi/2$.
\end{lemma}
\begin{proof}
  Since $t'(s)>0$, then $\phi''(s)=0$ if, and only if,
  $k_1(s)=-\frac{\phi''(s)}{t'(s)}=0$. By Proposition~\ref{prop:sign},
  this is equivalent to $k_2(s)=0$, which only holds when
  $\phi(s)=\pi/2$, since $k_2(s)=t'(s)\cot(\phi(s))$.
\end{proof}

The following Lemma is crucial in the description of rotational
special Weingarten surfaces of minimal type in $\esf^2\times\R$
without boundary. Recall that, in the case $\Sigma_\g$ has no
boundary, then $I=\R$ and $t(\R)=\R$.

\begin{lemma}
  \label{lem:minmax}
  Let $\Sigma_\g$ be a rotational special Weingarten surface as
  above. Assume $\Sigma_\g$ has no boundary. Then there exist
  $s_1,s_2\in\R$, $s_1<s_2$, such that $s_1$ is a local minimum
  of~$\phi$, $s_2$ is a local maximum of $\phi$, and
  \[
  \phi(s)=\phi(2s_1-s)=\phi(2s_2-s)\quad\mbox{and}\quad
  t(s)=2t(s_1)-t(2s_1-s)=2t(s_2)-t(2s_2-s)
  \]
  for any $s\in\R$.  In particular, $\phi$ is a periodic function of
  period
  \[
  T=2(s_2-s_1) ,
  \]
  and we can choose $s_1,s_2$ such that $\phi'(s)>0$ for any
  $s\in(s_1,s_2)$ and $\phi'(s)<0$ for any $s\in(s_2,s_1+T)$.
\end{lemma}
\begin{proof}
  Once proved there exist $s_1,s_2\in\R$ such that $s_1$ is a local
  minimum of $\phi$ and $s_2$ is a local maximum of $\phi$, let us see
  we have finished Lemma~\ref{lem:minmax}:  By Lemma~\ref{symmetry},
  \[
  \phi(s)=\phi(2s_1-s)=\phi(2s_2-s)\quad\mbox{and}\quad
  t(s)=2t(s_1)-t(2s_1-s)=2t(s_2)-t(2s_2-s)
  \]
  for any $s\in\R$.  We deduce that $\phi$ is a periodic function of
  period $T=2|s_2-s_1|$ and
  \[
  \phi(s_1)\leq\phi(s)\leq\phi(s_2)
  \]
  for any $s\in\R$. By periodicity, we can choose $s_1<s_2$. And
  taking nearer ones, we can assume $\phi'(s)>0$ for any
  $s\in(s_1,s_2)$. Since $\phi$ is an even function with respect to
  $s_2$, we deduce $\phi'(s)<0$ for any $s\in(s_2,s_1+T)$. This proves
  Lemma~\ref{lem:minmax} under the assumption there exist a local
  maximum and a local minimum of $\phi$.

  Let us suppose there does not exist a local maximum or a local
  minimum of $\phi$. Then, as $0\leq\phi\leq\pi$, there exists
  $\lim_{s\to+\infty}\phi(s)$.  Moreover we know that $\phi$ is
  strictly monotone in $(s_0,+\infty)$, for $s_0\in\R$ big
  enough and we can assume $\phi(s)\neq\pi/2$ in $(s_0,+\infty)$.
  By Lemma~\ref{lem:signS2}, $\phi''$ does not change sign in
  $(s_0,+\infty)$.  Then Lemma~\ref{lem:asymp} ensures there
  exists $\lim_{s\to +\infty} \phi'(s)=\pm 1$.
  On the other hand, consider $x_1\in(s_0,+\infty)$, and define
  $x_n=x_1+n-1$, for any $n\in\N$.  By the intermediate value theorem,
  we get the existence of $u_n\in(x_n,x_{n+1})$ such that
  \[
  \phi'(u_n)=\phi(x_{n+1})-\phi(x_n).
  \]
  Since there exists $\lim_{s\to+\infty}\phi(s)\in[0,\pi]$, then
  $\phi'(u_n)$ converges to zero as $n\to+\infty$, in contradiction
  with $\lim_{s\to +\infty} \phi'(s)=\pm 1$.
\end{proof}

We know from Lemma~\ref{lem:minmax} that $\Sigma_\g$ is symmetric with
respect to the horizontal slices $\esf^2\times\{t(s_1)+n\widetilde
T\}$ and $\esf^2\times\{t(s_2)+n\widetilde T\}$ for any $n\in\N$,
where $\widetilde T=2(t(s_2)-t(s_1))$.

We deduce from the following Lemma that, if
$\partial\Sigma_\g=\emptyset$, then $\Sigma_\g$ cannot be contained in
$(\esf^2)^+\times\R$, where $(\esf^2)^+$ denotes a hemisphere of
$\esf^2$.

\begin{lemma}\label{lem:s2}
  Let $\Sigma_\g$ be a rotational special Weingarten surface as
  above. Assume $\Sigma_\g$ is not contained in the vertical cylinder
  $\{\phi=\pi/2\}$. If $s_0$ is a local minimum (resp.  a local
  maximum) of $\phi$, then $\phi(s_0)<\pi/2$ and $\phi''(s_0)>0$
  (resp. $\phi(s_0)>\pi/2$ and $\phi''(s_0)<0$).
\end{lemma}
\begin{proof}
  Suppose $s_0$ is a local minimum of $\phi$. Then $\phi'(s_0)=0$
  (thus $t'(s_0)=1$) and $\phi''(s_0)\geq 0$.  In particular,
  $k_1(s_0)=-\phi''(s_0)\leq 0$. By Proposition~\ref{prop:sign},
  $k_2(s_0)\geq 0$.  Since $k_2(s_0)=\cot(\phi(s_0))$, we obtain
  $\phi(s_0)\leq\pi/2$, and the equality holds if, and only if,
  $\phi''(s_0)=0$, as we already remarked in Lemma~\ref{lem:signS2}.

  It cannot be $\phi(s_0)=\pi/2$ since, by uniqueness of the
  corresponding initial value problem~\eqref{eq:system}, $\phi$ would
  be constantly $\pi/2$. But we are assuming this is not the case.
  Then $\phi(s_0)<\pi/2$, and $\phi''(s_0)>0$.

  Following the same arguments we obtain $\phi(s_0)>\pi/2$ and
  $\phi''(s_0)<0$ when $s_0$ is a local maximum of~$\phi$. This
  finishes the proof of Lemma~\ref{lem:s2}.
\end{proof}

As a conclusion of all these previous Lemmas, we get the following
description of examples without boundary.

\begin{theorem}
  \label{th:s2}
  Let $\Sigma_\g$ be a rotational special Weingarten surface of
  minimal type without boundary in $\esf^2\times\R$, obtained by
  rotating a curve $\g(s)=(\sin(\phi(s)), 0,\cos(\phi(s)), t(s))$,
  $s\in I\subset\R$, around the axis $\{(0,0,1)\}\times\R$. Then
  $\Sigma_\g$ is a horizontal slice $\esf^2\times\{t_0\}$ or $I=\R$
  and either $\Sigma_\g$ is the vertical cylinder $\{\phi=\pi/2\}$ or
  an unduloidal type surface whose functions $\phi,t$ satisfy:
  \begin{itemize}
  \item There exist $s_1,s_2\in\R$, $s_1<s_2$, such that $\phi$ is a
    periodic function of period $T=2(s_2-s_1)$,
    $\phi(s_1)<\pi/2<\phi(s_2)$ and
    \[
    \phi(s_1)<\phi(s)<\phi(s_2)
    \]
    for any $s\in(s_1,s_1+T)$, $s\neq s_2$. Moreover, $\phi$ is an
    symmetric function with respect to $s_1,s_2$, i.e.
    \[
    \phi(s_1+s)=\phi(s_1-s)\quad\mbox{ and }\quad
    \phi(s_2+s)=\phi(s_2-s)
    \]
    for any $s\in[0,T)$.  Finally, there exists a unique
    $s_3\in(s_1,s_2)$ such that $\phi(s_3)=\pi/2$ and it holds
    $\phi''(s_3)=\phi''(2s_1-s_3)=0$, $\phi''>0$ in $(2s_1-s_3,s_3)$
    and $\phi''<0$ in $s\in(s_3,2s_2-s_3)$.
  \item The function $t$ is strictly monotone and
    $t(s)=2t(s_1)-t(2s_1-s)=2t(s_2)-t(2s_2-s)$ for any $s\in\R$. In
    particular,
    \[
    t(s+nT)=t(s)+n\widetilde T,\quad\mbox{for any }\ n\in\N,
    \]
    where $\widetilde T=2(t(s_2)-t(s_1))$.
  \end{itemize}
\end{theorem}


\section{Existence of examples in $\M_\e \times \R$}
\label{sec:sufficient}

Let $\Sigma_\g$ denote a rotational special Weingarten surface
without boundary of minimal type in $\M_\ve\times\R$, $\ve=\pm 1$,
obtained by rotating a curve
\[
\g(s)=(S_\ve(\phi(s)), 0,C_\ve(\phi(s)), t(s)),\quad s\in I\subset\R ,
\]
around the axis $\{(0,0,1)\}\times\R$.  Assume $t$ is a non-constant
function (i.e. $\Sigma_\g$ is not contained in a horizontal slice)
and $s$ is the arc-length parameter (i.e. $\phi'(s)^2+t'(s)^2=1$).
We know by Lemma~\ref{lem:no0} that either $t'>0$ or $t'<0$.  As
$\Sigma_\g$ is a Weingarten surface and $t'$ never vanishes,
$G(\phi,t',\phi'')=0$, where $G$ was defined in~\eqref{eq:G}. Also
we know by Theorems~\ref{th:h2} and~\ref{th:s2} that, when
$\partial\Sigma_\g\neq\emptyset$, there exists a (global) minimum
$s_0\in I=\R$ of $\phi$.

\begin{lemma}
\label{lem:-infinity}
Let $\Sigma_\g\subset\M_\ve\times\R$ denote a rotational special
Weingarten surface as above. Assume there is a local minimum $s_0\in
I$ of $\phi$, and denote $\phi_0=\phi(s_0)$.
\begin{enumerate}
\item If $t'>0$, then $\lim_{r \to -\infty} (r-f(r^2))<-\eta_\ve(
  \phi_0)$.\medskip
\item If $t'<0$, then $\lim_{r\to+\infty}(r-f(r^2))>\eta_\ve(\phi_0)$.
\end{enumerate}
\end{lemma}

\begin{proof}
  Since $s_0$ is a local minimum of $\phi$, then $\phi'(s_0)=0$,
  $t'(s_0)=\pm 1$ and $\phi''_0:=\phi''(s_0)\geq 0$.

  If $t'>0$, we have $t'(s_0)=1$.  Therefore,
  \[
  0=G(\phi_0,1,\phi''_0)=\eta_\e(\phi_0)+r_0-f({r_0}^2) ,
  \]
  where $r_0=-\frac{\eta_\e(\phi_0)+\phi''_0}2<0$.  We know, by
  Lemma~\ref{lem:increasing}, that the function $r-f(r^2)$ is strictly
  increasing in $r$, and then
  \[
  \lim_{r \to -\infty } (r-f(r^2))< r_0-f({r_0}^2)=- \eta_\e( \phi_0).
  \]
  This finishes item {\it (1)}.

  Suppose now $t'<0$, and thus $t'(s_0)=-1$. Arguing as above, we get
  \[
  0=G(\phi_0,-1,\phi''_0)=-\eta_\e(\phi_0)+\bar r_0-f({\bar r_0}^2) ,
  \]
  where $\bar r_0=\frac{\eta_\e(\phi_0)+\phi''_0}2>0$, and then
  \[
  \eta_\e(\phi_0)= \bar r_0-f({\bar r_0}^2) <\lim_{r\to
    +\infty}(r-f(r^2)).
    \]
\end{proof}

\begin{theorem}
  \label{th:exist1}
  Let $f\in {\mathcal C}^1([0,+\infty))$ be an elliptic function
  (i.e. $4x(f'(x))^2<1$ for any $x\geq 0$) such that $f(0)=0$.
  Given $\phi_0>0$ when $\e=-1$ or $\phi_0\in(0,\pi/2)$ when $\e=1$,
  satisfying
  \begin{equation}
    \label{condizione1}
    \eta_\e(\phi_0)<-\lim_{r\to -\infty}(r-f(r^2)),
  \end{equation}
  there exists a unique rotational special Weingarten surface 
  without boundary of minimal type in $\M_\ve\times\R$, obtained by
  rotating a curve $\g(s)=(S_\ve(\phi(s)), 0, C_\ve(\phi(s)), t(s))$,
  with $s\in\R$, parameterized by arc-length and such
  that $t(0)=0$, $t'>0$ and $\phi$ takes its minimum value $\phi_0$ at
  $s=0$.
\end{theorem}

\begin{proof}
  Picard-Lindel\"{o}f Theorem ensures there is a unique solution
  $(\phi,t)$ to the initial value problem~\eqref{eq:system.2} with
  $\phi(0)=\phi_0$, $\phi'(0)=0$, $t(0)=0$ and $t'(0)=1$; $(\phi,t)$
  defined in $(-s_1,s_1)$, for some $s_1>0$.

  In particular, $G(\phi_0,1,\phi''(0))=0$ (where $G$ was defined
  in~\eqref{eq:G}).  By Claim~\ref{lem:G}, $G(\phi_0,1,z)$ is strictly
  decreasing with respect to the $z$ variable. Moreover,
  Lemma~\ref{lem:increasing} assures
  $G(\phi_0,1,0)=\frac{\eta_\ve(\phi_0)}2-f
  \left(\frac{\eta_\ve(\phi_0)^2}4\right)>0$, since
  $\eta_\ve(\phi_0)>0$, and
  \[
  \lim_{z\to +\infty}G(\phi_0,1,z) = \lim_{z\to +\infty}
  \left(\eta_\ve(\phi_0)+\frac{(-z-\eta_\ve(\phi_0))}2-
    f\left(\frac{(-z-\eta_\ve(\phi_0))^2}4\right)\right)
  \]
  \[
  =\eta_\ve(\phi_0)+\lim_{r\to -\infty}(r-f(r^2))<0
  \]
  by~\eqref{condizione1}. Hence there exists a unique $\phi''_0\in\R$
  such that $G(\phi_0,1,\phi''_0)=0$ (so it must be
  $\phi''(0)=\phi''_0$) and $\phi''_0>0$. This says $\phi$ has a local
  minimum at $s=0$.

  By Lemma~\ref{symmetry}, the maximal interval of definition of the
  solution $(\phi,t)$ is of the kind $(-a,a)$, for some
  $a\in(0,+\infty]$ (see Remark~\ref{rem:sym}).  The proof will be
  complete once we will have shown that $a=+\infty.$

  Let us suppose $a<+\infty$. The functions $t,t',\phi,\phi'$ are
  continue on $(-a,a)$, so there exist their left limit values
  $t_a,t'_a,\phi_a,\phi'_a$ as $s \to a$ with $s<a$, respectively.
  Observe they are all finite: $t'_a,\phi'_a\in[-1,1]$ and
  $t_a,\phi_a<+\infty$ because $s$ was the arc-length parameter of
  $\g$ and $a<+\infty$.  By Lemma \ref{lem:no0}, $t'$ never vanishes
  in $(-a,a)$. Since $t'(0)=1$, we deduce $t'>0$ in $(-a,a)$, and
  then $t'_a\geq 0$.

  If $t'_a>0$ then we can solve the initial value
  problem~\eqref{eq:system.2} with initial conditions
  $\phi(a)=\phi_a$, $\phi'(a)=\phi'_a$, $t(a)=t_a$ and $t'(a)=t'_a$,
  extending the solution $(\phi,t)$ beyond $s=a$. That contradicts the
  assumption that the interval $(-a,a)$ is maximal.  Hence $t'_a=0$.
  But the maximum principle for surfaces with boundary will then say
  that $\Sigma\subset\M_\ve\times\{t_a\}$, where $\Sigma$ is the
  rotational special Weingarten surface of minimal type defined by
  $\phi,t$, which contradicts $t'(0)=1$. This contradiction proves
  Theorem~\ref{th:exist1}.
   \end{proof}

\begin{theorem}
  \label{th:exist2}
  Let $f\in {\mathcal C}^1([0,+\infty))$ be an elliptic function
  such that $f(0)=0$. Given $\phi_0>0$
  when $\e=-1$ or $\phi_0\in(0,\pi/2)$ when $\e=1$, satisfying
  \begin{equation}
    \label{cond1}
    \lim_{r\to +\infty}(r-f(r^2))>\eta_\e(\phi_0),
  \end{equation}
  there exists a unique rotational special Weingarten surface $\Sigma$
  without boundary of minimal type in $\M_\ve\times\R$ obtained by
  rotating a curve $\g(s)=(S_\ve(\phi(s)), 0, C_\ve(\phi(s)), t(s))$,
  with $s\in\R$, $\g$ parameterized by arc-length and such that
  $t(0)=0$, $t'<0$ and $\phi$ takes its minimum value $\phi_0$ at
  $s=0$.
\end{theorem}

\begin{proof}
  Picard-Lindel\"{o}f Theorem ensures there is a unique solution
  $(\phi,t)$ to the initial value problem~\eqref{eq:system.2} with
  $\phi(0)=\phi_0$, $\phi'(0)=0$, $t(0)=0$ and $t'(0)=-1$; $(\phi,t)$
  defined in $(-s_1,s_1)$, for some $s_1>0$.

  In particular, $G(\phi_0,-1,\phi''(0))=0$.  By
   Claim~\ref{lem:G},
  $G(\phi_0,-1,z)$ is strictly increasing with respect the $z$
  variable. Moreover, $G(\phi_0,-1,0)=-\frac{\eta_\ve(\phi_0)}2-f
  \left(\frac{\eta_\ve(\phi_0)^2}4\right)<0$ by
  Lemma~\ref{lem:increasing}, as $\eta_\ve(\phi_0)>0$, and
  \[
  \lim_{z\to +\infty}G(\phi_0,-1,z) = \lim_{z\to +\infty}
  \left(-\eta_\ve(\phi_0)+\frac{z+\eta_\ve(\phi_0)}2-
    f\left(\frac{(z+\eta_\ve(\phi_0))^2}4\right)\right)
  \]
  \[
  =-\eta_\ve(\phi_0)+\lim_{r\to +\infty}(r-f(r^2))>0
  \]
  by~\eqref{cond1}. Hence there exists a unique $\phi''_0\in\R$
  such that $G(\phi_0,-1,\phi''_0)=0$ (so it must be
$\phi''(0)=\phi''_0$)
and $\phi''_0>0$. This says $\phi$ has a local minimum at $0$.

We finish as in Theorem~\ref{th:exist1}.
\end{proof}

  Theorems~\ref{th:h2} and~\ref{th:s2} describe the rotational special
  Weingarten surfaces obtained in Theorems~\ref{th:exist1}
  and~\ref{th:exist2}.

\begin{theorem}[Classification of special embedded rotational surfaces
  of minimal type]
  \label{th:classify}
  Let $\Sigma$ be a Weingarten rotational special surface of minimal
  type in $\M_\ve \times
  \R$, possibly with boundary.
  \begin{itemize}
  \item If $\e=-1,$ then $\Sigma$ is contained in a horizontal slice
    $\H^2\times\{t_0\}$ or in a catenoidal type surface described by
    Theorem~\ref{th:h2}.
  \item When $\e=1$, $\Sigma$ is contained in a horizontal slice
    $\esf^2\times\{t_0\}$, in the vertical cylinder $\{\phi=\pi/2\}$
    or in a unduloidal type surface described by Theorem~\ref{th:s2}.
  \end{itemize}
\end{theorem}
\begin{proof}
  Assume that $\Sigma$ is generated by $\g(s)=(S_\ve(\phi(s)),
  0,C_\ve(\phi(s)),t(s))$ where $s \in I$ is the arc-length parameter
  and $I\subset\R$ an open interval.

  If $t$ is constant, then $\Sigma$ is contained in a horizontal
  slice.  If $\phi$ is a constant function, then Lemma~\ref{lem:cyl}
  says $\M_\ve=\esf^2$, and $\Sigma$ is contained in the vertical
  cylinder $\{\phi(s)=\pi/2\}$.  Hereafter we can assume $\phi$ and
  $t$ are not constant.

  If $I=\R$, then we can deduce the surface has no boundary, and the
  result follows from Theorems~\ref{th:s2} and~\ref{th:h2}.  Then
  assume $I=(a,b)\neq\R$.  But we can argue as in Theorem
  \ref{th:exist1} to extend the functions $\phi,t$ to $\R$, which
  proves the theorem.
\end{proof}

\begin{remark}
  R. Sa Earp and E. Toubiana \cite{ST1} showed that in $\R^3$ there
  exist rotational special Weingarten surfaces of minimal type with
  boundary which are not contained in a complete example without
  boundary (the only possible complete examples without boundary in
  $\R^3$ are of catenoidal type).  Theorem \ref{th:classify} shows
  such surfaces do not exist in $\M_\e\times\R$.
\end{remark}

\begin{remark}
  If the elliptic function $f$ vanishes identically, then Theorems
  \ref{th:exist1} and \ref{th:exist2} provides existence and
  uniqueness results for rotational minimal surfaces in
  $\M_\ve\times\R.$ Observe that, if $\M_\e=\esf^2,$ the hypotheses of
  such theorems are satisfied for every $\phi_0 \in (0,\pi/2).$
  Finally Theorem \ref{th:classify} gives the classification of those
  surfaces.
\end{remark}

\end{document}